\documentclass{amsart}
\usepackage[bookmarksnumbered,colorlinks, plainpages]{hyperref}
\hypersetup{colorlinks=true,linkcolor=red, anchorcolor=green, citecolor=cyan, urlcolor=red, filecolor=magenta, pdftoolbar=true}
\usepackage{cite}

\usepackage{subfig}
\usepackage{tikz}
\usetikzlibrary{arrows}

\usepackage{array}

\usepackage{amsmath, amsthm, amscd, amsfonts, amssymb, graphicx, color, enumerate, xcolor, mathrsfs, latexsym}

\newcommand{\Z}{\mathbb{Z}}

\newcommand{\C}{\mathcal{C}}
\newcommand{\F}{\mathcal{F}}
\newcommand{\I}{\mathcal{I}}
\newcommand{\M}{\mathcal{M}}
\renewcommand{\P}{\mathcal{P}}
\newcommand{\Q}{\mathcal{Q}}
\newcommand{\T}{\mathcal{T}}

\renewcommand{\S}{\mathbf{S}}
\newcommand{\1}{\mathbf{1}}

\renewcommand{\leq}{\leqslant}
\renewcommand{\geq}{\geqslant}

\newcommand{\ssp}{\hspace{1.8mm}}

\numberwithin{equation}{section}

\newtheorem{theorem}{Theorem}[section]
\newtheorem{lemma}[theorem]{Lemma}
\newtheorem{proposition}[theorem]{Proposition}
\newtheorem{corollary}[theorem]{Corollary}

\theoremstyle{definition}
\newtheorem{conjecture}[theorem]{Conjecture}

\newtheorem{example}[theorem]{Example}

\theoremstyle{remark}

\newtheorem*{acknowledgment}{Acknowledgment}

\begin{document}
\title{Lattice paths inside a table I}

\author[D. Yaqubi]{D. Yaqubi}
\address{Faculty of Agriculture and Animal Science, University of Torbat-e Jam, Iran.}
\email{daniel\_yaqubi@yahoo.es}

\author[M. Farrokhi D. G.]{M. Farrokhi D. G.}
\address{Department of Mathematics, Institute for Advanced Studies in Basic Sciences (IASBS), and the Center for Research in Basic Sciences and Contemporary Technologies, IASBS, Zanjan 66731-45137, Iran}
\email{m.farrokhi.d.g@gmail.com\\farrokhi@iasbs.ac.ir}

\author[H. Ghasemian Zoeram]{H. Ghasemian Zoeram}
\address{IDSIA, Lugano, Switzerland}
\email{hamed@idsia.ch}


\begin{abstract}
A lattice path in $\Z^d$ is a sequence $\nu_1,\nu_2,\ldots,\nu_k\in\Z^d$ such that the steps $\nu_i-\nu_{i-1}$ lie in a subset $\S$ of $\Z^d$ for all $i=2,\ldots,k$. Let $T_{m,n}$ be the $m\times n$ table in the first area of the $xy$-axis and put $\S=\{(1,1),(1,0),(1,-1)\}$. Accordingly, let $\I_m(n)$ denote the number of lattice paths starting from the first column and ending at the last column of $T$. We will study the numbers $\I_m(n)$ and give explicit formulas for special values of $m$ and $n$. As a result, we prove a conjecture of \textit{Alexander R. Povolotsky} involving $\I_n(n)$. Finally, we present some relationships between the number of lattice paths and Fibonacci and Pell-Lucas numbers, and pose an open problem.
\end{abstract}

\maketitle
\section{Introduction}
A \textit{lattice path} $L$ in $\Z^d$ is any sequence $\nu_1,\nu_2,\ldots,\nu_k$ of points of $\Z^d$ (see \cite{ck,ck-sgm}). The vectors $\nu_2-\nu_1,\nu_3-\nu_2,\ldots,\nu_k-\nu_{k-1}$ are called the \textit{steps} of $L$. Lattice paths are studied by fixing a set of steps and an area $U\subseteq\Z^d$ where the paths live in. A typical problem to carry out is to count possible lattice paths in the given area $U$ and a given length with steps in a given set $\S\subseteq\Z^d$.

Lattice paths and more generally lattice animals have deep roots in Physics and appear in the study of thermodynamic models, phase transitions, statistical physics, lattice gas models, river networks, etc. (see \cite{sf-yv} for example). A typical problem there is modeling a physical phenomena, say the motion of gas molecules, as paths inside a (triangular, square, hexagonal, etc.) lattice and study the behavior of the paths. A main question to ask is to give exact formulas or asymptotic results for the number of lattice paths (animals) satisfying some constraints. For example, it is shown that the number $a_n$ of directed animals of size $n$ satisfies 
\[a_n\sim\mu^nn^{-\theta}\]
for some constants $\mu$ and $\theta$ in various models. For a through study of $2$-dimensional lattices in Physics we refer the interested reader to \cite{ab12, vkb-hlb-ys, mb, arc, arc-rb-ajg, dd82, dd-mkp-mb, jamsd, dg-gv, vh-jpn, pgm, jpn-bd-jv, ijz}, and to \cite{dd83, th-gs, sl-sm, ymm-gs11, ymm-gs13, pgm, mkw, ijz} for higher dimensions.  We also refer to \cite{ab13, eb-pr-rs:94, eb-ad-ep-rp, gf-vp, yl-jm, sr-zry, hjr, sgw-ces} for further results. Gouyou-Beauchamps and Viennot \cite{dg-gv} give a bijection between compact-rooted directed lattice animals on two-dimensional square lattice with some lattice paths in the plane. Later Bousquet-M\'{e}lou and Conway \cite{mb-arc} and Corteel, Denise, and Gouyou-Beauchamps \cite{sc-ad-dg} give bijective proofs to obtain algebraic equations satisfied by area generating function of directed lattice animals on infinite families of two-dimensional lattices. Recall that a \textit{lattice animal} is a set of points in a lattice, which is a union of some lattice paths starting from a single point (or set of points in some contexts).

Lattice paths also arise naturally in various problems in mathematics and are well-studied in the literature. The general theory studies the analytic behavior of the complex generating function of the paths as well as giving estimations of the number of paths of a given length, etc. (see \cite{cb-pf, pf-rs,wp-wk} for instance). Particular lattice paths have received much attentions and studied extensively. The very important paths to mention are Dyck paths and Motzkin paths. A \textit{Dyck path} is a lattice path in $\Z^2$ starting from $(0,0)$ and ending at a point $(2n,0)$ ($n\geq0$) consisting of up-steps $(1,1)$ and down-steps $(1,-1)$, which never passes below the $x$-axis. The \textit{Catalan numbers} $\C_n=\frac{1}{n+1}\binom{2n}{n}$, a ubiquity in various combinatorial problems, count the number of Dyck paths of length $2n$ (see \cite{rpg, tk09, tk14, sr, rs} for details). Allowing the right steps $(1,0)$ in addition to those of a Dyck path, we get Motzkin paths starting from $(0,0)$ and ending at a point $(n,0)$, which never pass below the $x$-axis.

Throughout this paper, $T_{m,n}$ stands for the $m\times n$ table in the first quadrant composed of $mn$ unit squares, whose $(x,y)$-cell is located in the $x^{th}$-column from the left side and the $y^{th}$-row from the bottom side of $T_{m,n}$. Also, for a set $\S\subseteq\Z^d$ of steps, $l((i,j)\to (s,t);\S)$ denotes the number of all lattice paths in $T_{m,n}$ starting form the $(i,j)$-cell and ending at the $(s,t)$-cell with steps in $\S$, where $1\leq i,s\leq n$ and $1\leq j,t\leq m$. 

The paths we shall study in this paper use the same set $\S=\{(1,1),(1,0),(1,-1)\}$ of steps as Motzkin paths but live in a bounded rectangular area, which we may assume to be $T_{m,n}$. Notice that the number $l((1,1)\to (n,1);\S)$ of all lattice paths in the table $T_{m,n}$ starting from the $(1,1)$-cell and ending at the $(n,1)$-cell using Motzkin steps namely is the \textit{$n^{th}$-Motzkin number} provided that $m\geq n$. The number of all lattice paths is denoted by $\I_m(n)$. Indeed, 
\[\I_m(n)=\sum_{i,j=1}^m l((1,i)\to(n,j);\S).\]
Figure \ref{Fig:1} shows the number of all lattice paths for $m=2$ and $n=3$. Clearly, $l((1,i)\to(n,j))=l((1,i')\to(n,j'))$ when $i+i'=m+1$ and $j+j'=m+1$.
 
\begin{figure}[h!]
\begin{tikzpicture}[scale=1]
\draw[step=1cm] (0,0) grid (3,2);
\draw[o-stealth] (0.5, 0.5) -- (1.5, 0.5);
\draw[o-stealth] (1.5, 0.5) -- (2.5, 0.5);
\end{tikzpicture}
\begin{tikzpicture}[scale=1]
\draw[step=1cm] (0,0) grid (3,2);
\draw[o-stealth] (0.5, 0.5) -- (1.5, 0.5);
\draw[o-stealth] (1.5, 0.5) -- (2.5, 1.5);
\end{tikzpicture}
\begin{tikzpicture}[scale=1]
\draw[step=1cm] (0,0) grid (3,2);
\draw[o-stealth] (0.5, 0.5) -- (1.5, 1.5);
\draw[o-stealth] (1.5, 1.5) -- (2.5, 0.5);
\end{tikzpicture}
\begin{tikzpicture}[scale=1]
\draw[step=1cm] (0,0) grid (3,2);
\draw[o-stealth] (0.5, 0.5) -- (1.5, 1.5);
\draw[o-stealth] (1.5, 1.5) -- (2.5, 1.5);
\end{tikzpicture}
\\
\begin{tikzpicture}[scale=1]
\draw[step=1cm] (0,0) grid (3,2);
\draw[o-stealth] (0.5, 1.5) -- (1.5, 1.5);
\draw[o-stealth] (1.5, 1.5) -- (2.5, 1.5);
\end{tikzpicture}
\begin{tikzpicture}[scale=1]
\draw[step=1cm] (0,0) grid (3,2);
\draw[o-stealth] (0.5, 1.5) -- (1.5, 1.5);
\draw[o-stealth] (1.5, 1.5) -- (2.5, 0.5);
\end{tikzpicture}
\begin{tikzpicture}[scale=1]
\draw[step=1cm] (0,0) grid (3,2);
\draw[o-stealth] (0.5, 1.5) -- (1.5, 0.5);
\draw[o-stealth] (1.5, 0.5) -- (2.5, 1.5);
\end{tikzpicture}
\begin{tikzpicture}[scale=1]
\draw[step=1cm] (0,0) grid (3,2);
\draw[o-stealth] (0.5, 1.5) -- (1.5, 0.5);
\draw[o-stealth] (1.5, 0.5) -- (2.5, 0.5);
\end{tikzpicture}
\caption{All lattice paths in $T_{2,3}$.}
\label{Fig:1}
\end{figure} 

We intend to evaluate $\I_m(n)$ for special cases of $(m,n)$. In section 2, we obtain $\I_m(n)$ when $m\geq n$. Also, we prove a conjecture of \textit{Alexander R. Povolotsky} posed in OEIS sequence \href{https://oeis.org/A081113}{A081113} . In section 3, we shall compute $\I_m(n)$ for small values of $m$, namely $m=1,2,3,4$ as well as presenting some results for $\I_5(n)$. Finally, we use Fibonacci and Pell-Lucas numbers to prove some relations concerning lattice paths.
\section{$\I_n(n)$ vs Alexander R. Povolotsky's conjecture}\label{S}
Let $\S:=\{(1,1),(1,0),(1,-1)\}$. For positive integers $1\leq i,t\leq m$ and $1\leq s\leq n$, the number of all lattice paths from the $(1,i)$-cell to the $(s,t)$-cell in the table $T=T_{m,n}$ is denoted by $\C^i(s,t)$, that is, $\C^i(s,t)=l((1,i)\to(s,t);\S)$. Also, we put
\[\C_{m,n}(s,t)=\sum_{i=1}^m\C^i(s,t).\]
In case we are working in a single table, say $T$ as above, to avoid confusion we may use simply notation $\C(s,t)$ for $\C_{m,n}(s,t)$. Also, we put $\C_n(s,t):=\C_{n,n}(s,t)$. Clearly, $\C(s,t)$ is the number of all lattice paths from first column to the $(s,t)$-cell of $T$. It is easy to see for $n\geq 2$
\[\C_n(n,n)=\C_n(n-1,n)+\C_n(n-1,n-1),\]
where $\C_1(1,1)=1, \C_2(2,2)=2, \C_3(3,3)=5, \C_4(4,4)=13, \ldots$. The values of $\C_n(n,n)$ is OEIS sequence \href{https://oeis.org/A005773}{A005773}, where $T$ is a square table. By the way, notice how the diagram for $\C_4(4,4)=13$ is
\[\begin{matrix}
   1 & 2 & 5 & 13 \\
   1 & 3 & 8 & 21 \\
   1 & 3 & 8 & 21 \\
   1 & 2 & 5 & 13  
\end{matrix}\]
where each entry is the sum of two or three entries in the preceding column.

By symmetry of the table $T$, we have $\C(s,t)=\C(s,t')$ when $t+t'=m+1$. Table \ref{Table:1} illustrates the values of $\C(6,t)$, for all $1\leq t\leq 6$, where the number in $(s,t)$-cell of $T$ determines the number $\C(s,t)$.
\begin{center}
\begin{table}[h!]
\begin{tabular}{|c|c|c|c|c|c|}
\cline{6-6}
\multicolumn{5}{c}{}\vline&$\C(6,t)$\\
\hline
$1$&$2$&$5$&$13$&$35$&$96$\\\hline
$1$&$3$&$8$&$22$&$61$&$170$\\\hline
$1$&$3$&$9$&$26$&$74$&$209$\\\hline
$1$&$3$&$9$&$26$&$74$&$209$\\\hline
$1$&$3$&$8$&$22$&$61$&$170$\\\hline
$1$&$2$&$5$&$13$&$35$&$96$\\\hline
\end{tabular}
\caption{Values of $\C(6,t)$}
\label{Table:1}
\end{table}
\end{center}

It is worth mentioning that the numbers $\C_n(n,n)$ coincide with the number of directed animals of size $n$ starting from a single point (see \cite{dg-gv}). The numbers $\C_n(n,n)$ appear is various other results, see for example \cite{mb, mb-ar, nb, arc-ajg, dd-mkp-mb}. Note also that Krattenthaler and Yaqubi \cite{ck-dy} compute determinants of some Hankel matrices involving $\C_n(x,y)$, which is of independent interest.
\begin{theorem}\label{I_n(n)}
For any positive integer $n$ we have
\[\I_n(n)=3\I_{n-1}(n-1)+3^{n-1}-2\C_{n-1}(n-1,n-1).\]
\end{theorem}
\begin{proof}
Let $T:=T_{n,n}$ and $T':=T_{n-1,n-1}$ with $T'$ in the left-bottom side of $T$. Clearly, the number of lattice paths of $T$ which never meet the $n^{th}$ row of $T$ is
\[\I_{n-1}(n)=3\I_{n-1}(n-1)-2\C_{n-1}(n-1,n-1).\]
To obtain the number of all lattice paths we must count those who meet the $n^{th}$-row of $T$, that is equal to $3^{n-1}$. Thus $\I_{n}(n)-\I_{n-1}(n)=3^{n-1}$, from which the result follows.
\end{proof}

Michael Somos in OEIS sequence \href{https://oeis.org/A005773}{A005773} gives the following recurrence relation for $\C_n(n,n)$.
\begin{theorem}\label{D(n,n)}
Inside the square $n\times n$ table we have
\[n\C_n(n,n)=2n\C_n(n-1,n-1)+3(n-2)\C_n(n-2,n-2).\]
\end{theorem}

Utilizing Theorems \ref{I_n(n)} and \ref{D(n,n)} for $\C_n(n,n)$, we can prove a conjecture of Alexander R. Povolotsky posed in OEIS sequence \href{https://oeis.org/A081113}{A081113} as follows. This identity has appeared first in \cite{eb-rp-rs:92}
\begin{conjecture}
The following identity holds for the numbers $\I_n(n)$.
\begin{multline*}
(n+3)\I_{n+4}(n+4)=27n\I_n(n)+27\I_{n+1}(n+1)\\
-9(2n+5)\I_{n+2}(n+2)+(8n+21)\I_{n+3}(n+3).
\end{multline*}
\end{conjecture} 
\begin{proof}
Put 
\begin{align*}
A&=(n+3)\I_{n+4}(n+4),\\
B&=(8n+21)\I_{n+3}(n+3),\\
C&=9(2n+5)\I_{n+2}(n+2),\\
D&=27\I_{n+1}(n+1),\\
E&=27n\I_{n}(n).
\end{align*}
Using Theorem \ref{I_n(n)}, we can write
\begin{align}
A=&(3n+9)\I_{n+3}(n+3)+(n+3)3^{n+3}-(2n+6)\C_{n+3}(n+3,n+3)\cr
=&(8n+21)\I_{n+3}(n+3)-(5n+12)\I_{n+3}(n+3)+(n+3)3^{n+3}\cr
&-(2n+6)\C_{n+3}(n+3,n+3)\cr
=&B+(n+3)3^{n+3}-(5n+12)\I_{n+3}(n+3)\cr
&-(2n+6)\C_{n+3}(n+3,n+3).\label{AB}
\end{align}
Utilizing Theorem \ref{I_n(n)} once more for $\I_{n+3}(n+3)$ and $\I_{n+2}(n+2)$ yields
\begin{align*}
A=&B+(n+3)3^{n+3}-(5n+12)3^{n+2}\\
&-(18n+45)\I_{n+2}(n+2)-(2n+6)\C_{n+3}(n+3,n+3)\\
&+(10n+24)\C_{n+2}(n+2,n+2)+(3n+9)\I_{n+2}(n+2)+(n+3)3^{n+3}\\
=&B-C-(5n+12)3^{n+2}-(2n+6)\C_{n+3}(n+3,n+3)\\
&+(10n+24)\C_{n+2}(n+2,n+2)+9n\I_{n+1}(n+1)\\
&+27\I_{n+1}(n+1)+(3n+9)3^{n+1}-(6n+18)\C_{n+1}(n+1,n+1).
\end{align*}
It can be easily shown that
\begin{align}
A=&B-C+D\cr
&+(n+3)3^{n+3}-(2n+6)\C_{n+3}(n+3,n+3)-(5n+12)3^{n+2}\cr
&+(10n+24)\C_{n+2}(n+2,n+2)+9n\I_{n+1}(n+1)\cr
&+(3n+9)3^{n+1}-(6n+18)\C_{n+1}(n+1,n+1).\label{ABC}
\end{align}
Replacing $9n\I_{n+1}(n+1)$ by $27n\I_{n}(n)+n3^{n+2}-18n\I_{n}(n)$ in \ref{ABC} gives
\begin{multline*}
A=B-C+D+E\\
-(2n+6)\C_{n+3}(n+3,n+3)+(10n+24)\C_{n+2}(n+2,n+2)\\
-18n\C_n(n,n)-(6n+18)\C_{n+1}(n+1,n+1).
\end{multline*}
Since the coefficient of $\C_{n+3}(n+3,n+3)$ is $2(n+3)$, it follow from Theorem \ref{D(n,n)} that
\begin{align*}
A=&B-C+D+E-(4n+12)\C_{n+2}(n+2,n+2)-18n\C_n(n,n)\\
&+(10n+24)\C_{n+2}(n+2,n+2)-(6n+6)\C_{n+1}(n+1,n+1)\\
&-(6n+18)\C_{n+1}(n+1,n+1)\\
=&B-C+D+E-(4n+12)\C_{n+2}(n+2,n+2)\\
&-(6n+6)\C_{n+1}(n+1,n+1)+18n\C_n(n,n)-18n\C_n(n,n)\\
&-(12n+24)\C_{n+1}(n+1,n+1)+(6n+18)\C_{n+1}(n+1,n+1)\\
=&B-C+D+E,
\end{align*}
as required.
\end{proof}
\begin{theorem}\label{I_m(n)}
Inside the $m\times n$ table we have
\begin{align}
\I_m(n)=m3^{n-1}-2\sum_{s=1}^{n-1}3^{n-s-1}\C(s,1).
\end{align}
\end{theorem}
\begin{proof}
Let $T:=T_{m,n}$. The number of all lattice paths from the first column to the last column is simply $n3^{n-1}$ if they are allowed to get out of $T$. Now we count all lattice paths that go out of $T$ in some steps. First observe that the number of lattice paths that leave $T$ from the bottom row equals to those leave $T$ from the the top row in the first times. Suppose a lattice path goes out of $T$ from the bottom in column $s$ for the first times. The number of all partial lattice paths from the first column to the $(s-1,1)$-cell is simply $\C(s-1,1)$, and every such path continues in $3^{n-s}$ ways until it reaches the last column of $T$. Hence we have $3^{n-s}\C(s-1,1)$ paths leave the table $T$ from the bottom in column $s$ for any $s=2,\ldots,n$. Hence, the number of lattice paths is simply
\begin{align*}
\I_m(n)&=m3^{n-1}-2\sum_{s=2}^n3^{n-s}\C(s-1,1)\\
&=m3^{n-1}-2\sum_{s=1}^{n-1}3^{n-s-1}\C(s,1),
\end{align*}
as required.
\end{proof}
\begin{example}
Let $T$ be the square $6\times 6$ table. In Table \ref{Table:1}, every cell represents the number of all lattice paths from first column to that cell. Summing up the last column yields
\[\I_6(6)=96+170+209+209+170+96=950.\]
Now, utilizing Theorem \ref{I_m(n)}, we calculate $\I_6(6)$ in another way, as follows:
\begin{align*}
\I_6(6)=&6\cdot3^{6-1}-2\left(3^{6-1-1}\C(1,1)+3^{6-2-1}\C(2,1)+3^{6-3-1}\C(3,1)\right.\\
&\left.+3^{6-4-1}\C(4,1)+3^{6-5-1}\C(5,1)\right)\\
=&1458-2\left(3^4\cdot1+3^3\cdot2+3^2\cdot5+3^1\cdot13+3^0\cdot35\right)=950.
\end{align*}
\end{example}

Remind that the number $l(1,1;n+1,1:\S)$ of lattice paths in $\Z^2$ that never slides below the $x$-axis, is the \textit{$n^{th}$-Motzkin number} ($n\geq0$), denoted by $\M_{	n}$. Motzkin numbers begin with $1,1,2,4,9,21,\ldots$ (see OEIS sequence \href{https://oeis.org/A001006}{A001006}) and can be expressed in terms of binomial coefficients and Catalan numbers via
\[\M_{n}=\sum_{k=0}^ {\lfloor\frac{n}2\rfloor}\binom{n}{2k}\C_k.\]
\textit{Trinomial triangles} are defined by the same steps $(1,1)$, $(1,-1)$, and $(1,0)$ (in our notation) with no restriction by starting from a fixed cell. The number of ways to reach a cell is simply the sum of three numbers in the adjacent previous column. The $k^{th}$-entry of the $n^{th}$ column is denoted by $\binom{n}{k}_2$, where columns start by $0$. The middle entries of the Trinomial triangle, namely $1,1,3,7,19,\ldots$ (see \href{https://oeis.org/A002426}{A002426}) are studied by Euler. Analogously, Motzkin triangle  are defined by recurrence sequence 
\[\T(n,k)=\T(n-1,k-2)+\T(n-1,k-1)+\T(n-1,k),\]
for all $1\leq k\leq n-1$ and satisfy
\[\T(n,n)=\T(n-1,n-2)+\T(n-1,n-1)\]
for all $n\geq1$ (see  \href{https://oeis.org/A026300}{A026300}).

Table \ref{Table:2} illustrates initial parts of the above triangles with Motzkin triangle in the left and trinomial triangle in the right. For a positive integer $1\leq s\leq n$, each entry of the column $\C_s(s,1)$ is the sum of all entries in the $s^{th}$-row in the rotated Motzkin triangle, that is, $\C_s(s,1)=\sum_{i=1}^s\T(s,i)$. For example,
\[\C(4,1)=\T(4,1)+\T(4,2)+\T(4,3)+\T(4,4)=4+5+3+1=13.\]
The entries in the first column of rotated Motzkin triangle are indeed the Motzkin numbers.
\begin{table}[h!]
\begin{center}
\begin{tabular}{>{$}c<{$}|*{7}{c}}
\multicolumn{1}{l}{$\C_s(s,1)$} &&&&&&&\\\cline{1-1} 
1 &1&&&&&&\\
2 &1&1&&&&&\\
5 &2&2&1&&&&\\
13 &4&5&3&1&&&\\
35 &9&12&9&4&1&&\\
96 &21&30&25&14&5&1\\
\multicolumn{1}{l}{} &\multicolumn{7}{c}{}
\end{tabular}
\begin{tabular}{rccccccccc}
 & & & & 1\\\noalign{\smallskip\smallskip}
 & & & 1 & 1 & 1\\\noalign{\smallskip\smallskip}
 & & 1 & 2 & 3 & 2 & 1\\\noalign{\smallskip\smallskip}
 & 1 & 3 & 6 & 7 & 6 & 3 & 1\\\noalign{\smallskip\smallskip}
 1 & 4 & 10 & 16 & 19 & 16 & 10 & 4 & 1\\\noalign{\smallskip\smallskip}
\end{tabular}
\end{center}
\caption{Motzkin triangle (left) and trinomial triangle (right) rotates $90^\circ$ clockwise}
\label{Table:2}
\end{table}
\begin{lemma}\label{D(s,1)}
Inside the square $n\times n$ table we have
\[\C_n(s,1)=3\C_n(s-1,1)-\M_{s-2},\]
for all $1\leq s\leq n$.
\end{lemma}
\begin{proof}
Let $T:=T_{n,n}$. By the definition, $\C(s,1)$ is the number of all lattice paths from the first column to $(s,1)$-cell. This number equals the number of lattice paths from $(s,1)$-cell to the first column with reverse steps that lie inside the table $T$, which is equal to $3^{s-1}$ minus those paths that leave $T$ at some point. Consider all those lattice paths staring from $(s,1)$-cell with reverse steps that leaves $T$ at $(i,0)$ for the first time, where $1\leq i\leq s-1$. Clearly, the number of such paths are $3^{i-1}\M_{s-i-1}$. Thus 
\[\C_n(s,1)=3^{s-1}-\sum_{i=1}^{s-1}3^{i-1}\M_{s-i-1},\]
from which it follows that $\C_n(s,1)=3\C_n(s-1,1)-\M_{s-2}$, as required.
\end{proof}
\begin{example}
Consider the Table \ref{Table:2}. Using Lemma \ref{D(s,1)} we can calculate $\C(6,1)$ as
\[\C_6(6,1)=3\C_6(5,1)-\M_4=3\cdot35-9=96.\]
\end{example}
\begin{corollary}\label{I_n(n)+detail}
Inside the $n\times n$ table we have
\[\I_n(n)=(n+2)3^{n-2}+2\sum_{k=0}^{n-3}(n-k-2)3^{n-k-3}\M_k.\]
\end{corollary}
\begin{proof}
The result follows from Theorem \ref{I_m(n)} and Lemma \ref{D(s,1)}.
\end{proof}

The next result shows that the number of lattice paths in $T_{m,n}$ is independent of the number $m$ of rows provided that $m$ is big enough.
\begin{theorem}\label{I_m+1(n)-I_m(n)}
Inside the $m\times n$ table ($m\geq n$) we have
\[\I_{m+1}(n)-\I_m(n)=\sum_{i=0}^{n-1}\C(i,1)\C(n-i,1),\]
where we assume that $\C(0,1)=1$.
\end{theorem}
\begin{proof}
Consider the table $T:=T_{m,n}$. We construct the table $T'$ by adding a new row $m+1$ at the top of $T$. Now to count the number of all lattice paths in $T'$, it is sufficient to consider lattice paths that reach to the new row $m+1$ for the first time. Assume a lattice path reaches to the row $m+1$ at column $i$ for the first time. Then its initial part from column $1$ to column $i-1$ is a lattice path from the first column of $T$ to $(i-1,m)$-cell. Also, its terminal part from column $i$ to column $n$ is a lattice path from $(i,m+1)$-cell of $T'$ to its last column, which is in one to one correspondence with a lattice path from $(i,m)$-cell of $T$ to its last column as $m\geq n$. Hence, the number of such paths is simply $\C(i-1,m)\C(n-i+1,m)$, which is equal to $\C(i-1,1)\C(n-i+1,1)$ by symmetry. Therefore
\[\I_{m+1}(n)-\I_m(n)=\sum_{i=1}^n\C(i-1,1)\C(n-i+1,1)\]
and the result follows.
\end{proof}
\begin{corollary}\label{I_m(n) long formula}
For $m\geq n$ we have
\begin{multline*}
\I_m(n)=(n+2)3^{n-2}+(m-n)\sum_{i=0}^{n-1}\C(i,1)\C(n-i,1)\\
+2\sum_{k=0}^{n-3}(n-k-2)3^{n-k-3}\M_k.
\end{multline*}
\end{corollary}
\begin{proof}
Let $m=n+k$, where $k$ is a positive integer. Then
\begin{align*}
\I_m(n)-\I_n(n)&=(\I_m(n)-\I_{m-1}(n))+\cdots+(\I_{n+1}(m)-\I_n(m))\\
&=(m-n)\sum_{i=0}^{n-1}\C(i,1)\C(n-i,1).
\end{align*}	
Now the result follows from Corollary \ref{I_n(n)+detail}.
\end{proof}
\begin{theorem}\label{C(i,n)C(n-i,n)}
Inside the $m\times n$ table with $m\geq 2n-2$ we have
\begin{itemize}
\item[(i)]$\sum_{i=0}^{n-1}\C(i,n)\C(n-i,n)=3^{n-1}$;
\item[(ii)]$\sum_{i=1}^{n-1}\C(i,n)\C(n-i,n)=\sum_{i=0}^{n-2}3^{n-i-1}\M_i$;
\item[(iii)]$\I_m(n)=(3m-2n+2)3^{n-2}+2\sum_{k=0}^{n-3}(n-k-2)3^{n-k-3}\M_k$.
\end{itemize}
\end{theorem}
\begin{proof}
(i) Let $T:=T_{m,n}$ with $m=2n-2$ and $T'$ be the table obtained by adding a new row in the middle of $T$. By Theorem \ref{I_m+1(n)-I_m(n)}, it is sufficient to obtain $\I_{m+1}(n)-\I_m(n)$. Clearly, the number of lattice paths reaching to any $(i,n)$-cell of $T$ or $T'$ is the same for all $i=1,\ldots,n-1$. On the other hand, the number of all lattice paths of $T'$ reaching at $(n,n)$-cell is $3^{n-1}$ since we may begin the paths form the last $(n,n)$-cell and apply reverse steps with no limitation until to reach the first column. Thus
\[3^{n-1}=\I_{m+1}(n)-\I_m(n)=\sum_{i=0}^{n-1}\C(i,1)\C(n-i,1).\] 
(ii) Put $\C(0,1)=1$. Then
\[\C(n,1)=3^{n-1}-\sum_{i=1}^{n-1}\C(i,1)\C(n-i,1).\]
On the other hand, by Lemma \ref{D(s,1)}, we have 
\[\C(n,1)=3^{n-1}-\sum_{i=0}^{n-2}3^{n-i-2}\M_i,\]
from which the result follows.

(iii) It follows from (i) and Corollary \ref{I_m(n) long formula}.
\end{proof}
\begin{lemma}
Inside the $n\times n$ table we have
\begin{align*}
\C_n(n,k+2)-\C_n(n,k)=\sum_{i=1}^{n-1}\left(\C_n(i,k+3)-\C_n(i,k-1)\right)
\end{align*}
for all $1\leq k\leq n$.
\end{lemma}
\begin{proof}
For $n=2$, the result is trivially true. For any $l<n$ we have
\begin{align*}
\C_n(l+1,k+2)&=\C_n(l,k+3)+\C_n(l,k+2)+\C_n(l,k+1)\\
\C_n(l+1,k)&=\C_n(l,k+1)+\C_n(l,k)+\C_n(l,k-1),
\end{align*}
which imply that
\[\C_n(l+1,k+2)-\C_n(l+1,k)=\C_n(l,k+3)-\C_n(l,k-1)+\left(\C_n(l,k+2)-\C_n(l,k)\right).\]
Thus
\[\C_n(n,k+2)-\C_n(n,k)=\sum_{i=1}^{n-1}\left(\C_n(i,k+3)-\C_n(i,k-1)\right)\]
as $\C_n(1,k+2)-\C_n(1,k)=0$. This completes the proof.
\end{proof}

Theorems \ref{I_m+1(n)-I_m(n)} and \ref{C(i,n)C(n-i,n)} give some formulas for the (convolution) product of an specific row with itself. Regarding columns, we get the following (more) general results.
\begin{theorem}\label{columns inner product}
Inside the $m\times n$ table, we have
\[\I_m(n)=\sum_{i=1}^m\C(a,i)\C(b,i)\]
for all $a,b\geq1$ such that $a+b=n+1$. In other words, the inner product of columns $a$ and $b$ equals $\I_m(n)$. In particular, if $n=2k-1$ is odd, then
\[\I_m(n)=\sum_{i=1}^m\C_{k,i}^2.\]
\end{theorem}
\begin{proof}
Every lattice path crosses the column $a$ at some row, say $i$. The number of such paths equals the number $\C(a,i)$ of paths from the first column to the $(a,i)$-cell multiplied by the number $\C(n-(a-1),i)=\C(b,i)$ of paths from the last column to that cell, from which the result follows.
\end{proof}
\section{Tables with few rows}
In this section, we shall compute $\I_m(n)$ for $m=1,2,3,4$ and arbitrary positive integers $n$. Also, we obtain some properties of $\I_m(n)$ for $m=5$. Some values of the $\I_3(n)$ and $\I_4(n)$ are already given in \href{https://oeis.org/A0011333}{A001333} and  \href{https://oeis.org/A055819}{A055819}, respectively.
\begin{lemma}
$\I_1(n)=1$ and $\I_2(n)=2^n$ for all $n\geq1$.
\end{lemma}

Let $x$ and $y$ be arbitrary real numbers. By the binomial theorem, we have the following identity, 
\[x^n+y^n=(x+y)^n+\sum_{k=1}^{\lfloor\frac{n}2\rfloor} (-1)^k\left[\binom{n-k}{k}+\binom{n-k-1}{k-1}\right](xy)^k(x+y)^{n-2k},\]
where $n\geq 1$. This identity also can rewritten as
\begin{equation}\label{Binomial}
x^n+y^n=\sum_{k=0}^{\lfloor\frac{n}2\rfloor} (-1)^k\left[\binom{n-k}{k}+\binom{n-k-1}{k-1}\right](xy)^k(x+y)^{n-2k},
\end{equation}
where $\binom{r}{-1}=0$. Pell-Lucas sequence \cite{tk14} is defined as $\Q_1=1$, $\Q_2=3$, and $\Q_n=2\Q_{n-1}+\Q_{n-2}$ for all $n\geq3$. It can also be defined by the so called \textit{Binet formula} as $\Q_n=(\alpha^n+\beta^n)/2$, where $\alpha=1+\sqrt{2}$ and $\beta=1-\sqrt{2}$ are solutions of the quadratic equation $x^2=2x+1$.
\begin{lemma}\label{I_3(n)}
For all $n\geq1$ we have $\I_3(n)=\Q_{n+1}$.
\end{lemma}
\begin{proof}
The number of lattice paths to cells in columns $n-2$, $n-1$, and $n$ of $T_{3,n}$ looks like
\begin{center}
\begin{tabular}{|c|c|c|}
\hline
$n-2$&$n-1$&$n$\\
\hline\hline
$x$&$x+y$&$3x+2y$\\
\hline
$y$&$2x+y$&$4x+3y$\\
\hline
$x$&$x+y$&$3x+2y$\\
\hline
\end{tabular}
\end{center}
which imply that $\I_3(n-2)=2x+y$, $\I_3(n-1)=4x+3y$, and $\I_3(n)=10x+7y$. Thus the following linear recurrence exists for $\I_3$.
\begin{align}\label{F}
\I_3(n)=2\I_3(n-1)+\I_3(n-2).
\end{align}
Since $\I_3(1)=\Q_2=3$ and $\I_3(2)=\Q_3=7$, it follows that $\I_3(n)=\Q_{n+1}$ for all $n\geq1$, as required.
\end{proof}
\begin{corollary}
 Let $n$ be a positive integer. Then
\[\I_3(n)=\sum_{k=0}^{\lfloor\frac{n+1}2\rfloor}\left[\binom{n-k+1}{k}+\binom{n-k}{k-1}\right]2^{n-2k}.\]
\end{corollary}
\begin{proof}
 It is sufficient to put $x=\alpha$ and $y=\beta$ in \eqref{Binomial}.
\end{proof}

The Fibonacci sequence \href{https://oeis.org/A000045}{A000045} starts with the integers $0$ and $1$, and every other term is the sum of the two preceding ones, that is, $\F_0=0$, $\F_1=1$, and $\F_n=\F_{n-1}+\F_{n-2}$ for all $n\geq2$. This recursion gives the Binet's formula $\F_n=\frac{\varphi^n-\psi^n}{\varphi-\psi}$, where $\varphi=\frac{1+\sqrt{5}}2$ and $\psi=\frac{1-\sqrt{5}}2$. 
\begin{lemma}\label{I_4(n)}
For all $n\geq1$ we have $\I_4(n)=2\F_{2n+1}$.
\end{lemma}
\begin{proof}
The number of lattice paths to cells in columns $n-2$, $n-1$, and $n$ of $T_{4,n}$ looks like
\begin{center}
\begin{tabular}{|c|c|c|}
\hline
$n-2$&$n-1$&$n$\\
\hline\hline
$x$&$x+y$&$2x+3y$\\
\hline
$y$&$x+2y$&$3x+5y$\\
\hline
$y$&$x+2y$&$3x+5y$\\
\hline
$x$&$x+y$&$2x+3y$\\
\hline
\end{tabular}
\end{center}
which imply that $\I_4(n-2)=2x+2y$, $\I_4(n-1)=4x+6y$, and $\I_4(n)=10x+16y$. Hence we get the following linear recurrence for $\I_4$.
\begin{align}\label{I4n}
\I_4(n)=3\I_4(n-1)-\I_4(n-2).
\end{align}
On the other hand,
\begin{align*}
\F_{2n+1}&=\F_{2n}+\F_{2n-1}\\
&=2\F_{2n-1}+\F_{2n-2}\\
&=3\F_{2n-1}-\F_{2n-3}\\
&=3\F_{2(n-1)+1}-\F_{2(n-2)+1}.
\end{align*}
Now since  $\I_4(1)=2\F_3$ and $\I_4(2)=2\F_5$, it follows that $\I_4(n)=2\F_{2n+1}$ for all $n\geq1$. The proof is complete.
\end{proof}
\begin{corollary}
For all $n\geq1$ we have
\begin{align}
\I_4(n)=\sum_{k=0}^{n}\left(-1\right)^k\left[\frac{2n+1}{k}\binom{2n-k}{k-1}\right]5^{n-k}.
\end{align}
\end{corollary}
\begin{proof}
It is sufficient to put $x=\varphi$ and $y=\psi$ in \eqref{Binomial}.
\end{proof}

In the sequel, we obtain some properties of $C_{m,n}(s,t)$ and $\I_m(n)$, when $m=5$.
\begin{proposition}\label{D(s+2,3)=2I_5(s)-1}
Inside the $5\times n$ table we have
\[\C(s+2,1)=\I_5(s)\quad\text{and}\quad\C(s+2,3)=2\I_5(s)-1\]
for all $1\leq s\leq n$.
\end{proposition}
\begin{proof}
From the table in Example \ref{coefficient matrix 5xn}, it follows simply that $\I_5(s)=\C(s+2,1)$ for all $s\geq1$. Also, from the table, it follows that 
\[2\C(s+1,1)-\C(s+1,3)=2\C(s,1)-\C(s,3)\]
for all $s\geq1$, that is, $2\C(s,1)-\C(s,3)$ is constant. Since $2\C(1,1)-\C(1,3)=1$, we get $2\C(s+2,1)-\C(s+2,3)=1$, from which the result follows.
\end{proof}
\begin{proposition}
Inside the $5\times n$ table we have
\[\C(s,1)\times\C(s+t,3)-\C(s,3)\times\C(s+t,1)=\sum_{i=s}^{s+t-1}\C(i,2)\]
for all $1\leq s,t\leq n$.
\end{proposition}
\begin{proof}
From Proposition \ref{D(s+2,3)=2I_5(s)-1}, we know that $\C(s,3)=2\C(s,1)-1$ for all $1\leq s\leq n$. Then
\begin{align*}
&\C(s,1)\C(s+t,3)-\C(s,3)\C(s+t,1)\\
=&\C(s,1)(2\C(s+t,1)-1)-(2\C(s,1)-1)\C(s+t,1)\\
=&2\C(s,1)\C(s+t,1)-\C(s,1)-2\C(s,1)\C(s+t,1)+\C(s+t,1)\\
=&\C(s+t,1)-\C(s,1).
\end{align*}
On the other hand,
\begin{align*}
\C(s+t,1)-\C(s,1)&=\C(s+t-1,1)+\C(s+t-1,2)-\C(s,1)\\
&=\C(s+t-2,1)+\C(s+t-2,2)+\C(s+t-1,2)-\C(s,1)\\
&\ssp\vdots\\
&=\sum_{i=s}^{s+t-1}\C(i,2)+\C(s,1)-\C(s,1)\\
&=\sum_{i=s}^{s+t-1}\C(i,2),
\end{align*}
from which the result follows.
\end{proof}

\section{Further results about lattice paths by using Fibonacci and Pell-Lucas numbers}
In this section, we obtain some relations and properties about lattice paths by the aid of Fibonacci and Pell-Lucas sequences.
\begin{proposition}\label{T(4,n) entries}
Inside the $4\times n$ table we have
\[\C(s,1)=\F_{2s-1}\quad\text{and}\quad\C(s,2)=\F_{2s}\]
for all $s\geq1$. As a result,
\[\C(s,1)\times\C(s+t,2)-\C(s,2)\times\C(s+t,1)=\C(s,2).\]
for all $s,t\geq1$.
\end{proposition}
\begin{proof}
Clearly $\C(1,1)=\C(1,2)=\F_1=\F_2=1$. Now since 
\begin{align*}
\C(s,1)&=\C(s-1,1)+\C(s-1,2),\\
\C(s,2)&=2\C(s-1,2)+\C(s-1,1).
\end{align*}
we may prove, by using induction that, $\C(s,1)=\F_{2s-1}$ and $\C(s,2)=\F_{2s}$ for all $s\geq1$. The second claim follows from the fact that
\[\F_{2s-1}\F_{2s+2t}-\F_{2s}\F_{2s+2t-1}=\F_{2s}.\]
The proof is complete.
\end{proof}
\begin{proposition}
Inside the $4\times n$ table we have
\[\I_4(2s+1)=\frac{1}{4}\I_4(s+1)^2+\C(s,2)^2\]
for all $1\leq s\leq n$.
\end{proposition}
\begin{proof}
Following Lemma \ref{I_4(n)} and Proposition \ref{T(4,n) entries}, it is enough to show that
\[2\F_{4s+3}=\F_{2s+3}^2+\F_{2s}^2.\]
First observe that the equation $\F_{2n-1}=\F_n^2+\F_{n-1}^2$ yields $\F_{4s+1}=\F_{2s+1}^2+\F_{2s+2}^2$ and $\F_{4s+5}=\F_{2s+3}^2+\F_{2s+2}^2$. Now, by combining these two formulas, we obtain
\begin{align*}
\F_{2s+3}^2+\F_{2s}^2&=\F_{4s+5}+\F_{4s+1}-(\F_{2s+1}^2+\F_{2s+2}^2)\\
&=\F_{4s+4}+\F_{4s+3}+\F_{4s+1}-\F_{4s+3}\\
&=\F_{4s+3}+\F_{4s+2}+\F_{4s+1}\\
&=2\F_{4s+3},
\end{align*}
as required.
\end{proof}

\textit{Pell numbers} $\P_n$ are defined recursively as $\P_1=1$, $\P_2=2$, and $\P_n=2\P_{n-1}+\P_{n-2}$ for all $n\geq3$. The Binet's formula corresponding to $\P_n$ is $\P_n=\frac{\alpha^n-\beta^n}{\alpha-\beta}$, where $\alpha=1+\sqrt{2}$ and $\beta=1-\sqrt{2}$. 
\begin{proposition}\label{T(3,n) entries}
Inside the $3\times n$ table we have
\[\C(s,1)=\P_s\quad\text{and}\quad\C(s,2)=\Q_s\]
for all $s\geq1$. As a result,
\[\C(s,1)\times\C(s+t,2)-\C(s,2)\times\C(s+t,1)=(-1)^{s+1}\C(t,1).\]
for all $s,t\geq1$.
\end{proposition}
\begin{proof}
From the table in Lemma \ref{I_3(n)}, we observe that 
\begin{align*}
\C(s,1)&=2\C(s-1,1)+\C(s-2,1),\\
\C(s,2)&=2\C(s-1,2)+\C(s-2,2)
\end{align*}
for all $s\geq3$. Now since $\C(1,1)=\P_1=1$, $\C(2,1)=\P_2=2$, $\C(1,2)=\Q_1=1$, and $\C(2,2)=\Q_2=3$ one can show, by using induction, that $\C(s,1)=\P_s$ and $\C(s,2)=\Q_s$ for all $s$.
To prove the second claim, we use the following formula
\[\P_s\Q_{s+t}-\Q_s\P_{s+t}=(-1)^{s+1}\P_t\]
that can be proved simply by using Binet's formulas.
\end{proof}
\section{Further work}
We end our paper with posing few open problems on determinant of matrices arising from lattice paths.

First consider the $m\times n$ table $T$ with $2n\geq m$. For positive integers $\ell_1,\ell_2,\ldots,\ell_{\lceil\frac{m}{2}\rceil}$, we can write $\I_m(n)$ as
\[\I_m(n)=\ell_1\I_m(n-1)+\ell_2\I_m(n-2)+\cdots+\ell_{\lceil\frac{m}{2}\rceil}\I_m(n-\lceil\frac{m}{2}\rceil).\]
Also, for positive integers $0\leq s\leq\lceil\frac{m}{2}\rceil$ and $k_{1,s},k_{2,s},\ldots,k_{\lceil\frac{m}{2}\rceil,s}$, we put
\[\I_m(n-s)=k_{1,s}x_1+k_{2,s}x_2+\cdots+k_{\lceil\frac{m}{2}\rceil,s}x_{\lceil\frac{m}{2}\rceil},\]
where $x_t=\C(n-\lceil\frac{m}{2}\rceil,t)=\sum_{i=1}^m\C^i(n-\lceil\frac{m}{2}\rceil,t)$ is the number of all lattice paths from the first column to the $(n-\lceil\frac{m}{2}\rceil,t)$-cell of $T$, for each $1\leq i\leq m$ and $1\leq t\leq\lceil\frac{m}{2}\rceil$. Utilizing the above notation, we can can write
\begin{align}\label{I_m(n) by sum of x_i}
\begin{aligned}
\I_m(n)=&k_{1,0}x_1+k_{2,0}x_2+\cdots+k_{\lceil\frac{m}{2}\rceil,0}x_{\lceil\frac{m}{2}\rceil}\cr
=&\ell_1\I_{n-1}+\ell_2\I_{n-2}+\cdots+\ell_{\lceil\frac{m}{2}\rceil}\I_{n-\lceil\frac{m}{2}\rceil}\cr
=&\ell_1(k_{1,1}x_1+k_{2,1}x_2+\cdots+k_{\lceil\frac{m}{2}\rceil,1}x_{\lceil\frac{m}{2}\rceil})\cr
&+\ell_2(k_{1,2}x_1+k_{2,2}x_2+\cdots+k_{\lceil\frac{m}{2}\rceil,2}x_{\lceil\frac{m}{2}\rceil})\cr
&\ \vdots\cr
&+\ell_{\lceil\frac{m}{2}\rceil}(k_{1,{\lceil\frac{m}{2}\rceil}}x_1+k_{2,\lceil\frac{m}{2}\rfloor}x_2+\cdots+k_{\lceil\frac{m}{2}\rceil,\lceil\frac{m}{2}\rceil}x_{\lceil\frac{m}{2}\rceil}).
\end{aligned}
\end{align}

From \eqref{I_m(n) by sum of x_i}, we obtain the following system of linear equations
\begin{equation}\label{system of equations}
\left\{
\begin{array}{ccccccc}
k_{1,1}\ell_1 &+&\cdots &+& k_{1,\lceil\frac{m}{2}\rceil}\ell_{\lceil\frac{m}{2}\rceil}&=&k_{1,0},\\
k_{2,1}\ell_1 &+&\cdots &+& k_{2,\lceil\frac{m}{2}\rceil}\ell_{\lceil\frac{m}{2}\rceil}&=&k_{2,0},\\
\vdots &\vdots &\ddots &\vdots &\vdots &\vdots &\vdots\\
k_{\lceil\frac{m}{2}\rceil,1}\ell_1 &+&\cdots &+& k_{\lceil\frac{m}{2}\rceil,\lceil\frac{m}{2}\rceil}\ell_{\lceil\frac{m}{2}\rceil}&=& k_{\lceil\frac{m}{2}\rceil,0}. 
\end{array}
\right.
\end{equation}
Now consider the following coefficient matrix $A$ of the system \eqref{system of equations}
\[A=\begin{bmatrix}
k_{1,1}&k_{1,2}&\cdots & k_{1,\lceil\frac{m}{2}\rceil}\\
k_{2,1}&k_{2,2}&\cdots & k_{2,\lceil\frac{m}{2}\rceil}\\
\vdots&\vdots&\ddots&\vdots\\
k_{\lceil\frac{m}{2}\rceil,1}&k_{\lceil\frac{m}{2}\rceil,2}&\cdots&k_{\lceil\frac{m}{2}\rceil,\lceil\frac{m}{2}\rceil}
\end{bmatrix},\]
which we call the \textit{coefficient matrix} of the table $T$ and denote it by $\C(T)$.
\begin{conjecture}
For a given $m\times n$ table $T$ ($2n\geq m$), we have $\det(\C(T))=-2^{\lfloor\frac{m}{2}\rfloor}$.
\end{conjecture}

\begin{example}\label{coefficient matrix 5xn}
Let $T$ be a $5\times n$ table. The columns $n-3$, $n-2$, $n-1$, and $n$ of $T$ are given by
\begin{center}
\begin{tabular}{|c|c|c|c|}
\hline
$n-3$&$n-2$&$n-1$&$n$\\
\hline\hline
$x_1$&$x_1+x_2$&$2x_1+2x_2+x_3$&$4x_1+6x_2+3x_3$\\
\hline
$x_2$&$x_1+x_2+x_3$&$2x_1+4x_2+2x_3$&$6x_1+10x_2+6x_3$\\
\hline
$x_3$&$2x_2+x_3$&$2x_1+4x_2+3x_3$&$6x_1+12x_2+7x_3$\\
\hline
$x_2$&$x_1+x_2+x_3$&$2x_1+4x_2+2x_3$&$6x_1+10x_2+6x_3$\\
\hline
$x_1$&$x_1+x_2$&$2x_1+2x_2+x_3$&$4x_1+6x_2+3x_3$\\
\hline
\end{tabular}
\end{center}
from which it follows that
\begin{align*}
\I_5(n-3)&=2x_1+2x_2+x_3,\\
\I_5(n-2)&=4x_1+6x_2+3x_3,\\
\I_5(n-1)&=10x_1+16x_2+9x_3,\\
\I_5(n)&=28x_1+44x_2+25x_3
\end{align*}
Clearly,
\[\I_5(n)=\ell_1\I_5(n-1)+\ell_2\I_5(n-2)+\ell_3\I_5(n-3)\]
for some $\ell_1,\ell_2,\ell_3$, and that the coefficient matrix of the table $T$ is $\C(T)=\begin{bmatrix}10&4&2\\16&6&2\\9&3&1\\\end{bmatrix}$. It is obvious that $\det(\C(T))=-2^{\lfloor\frac{5}{2}\rfloor}=-4$.
\end{example}

Our second problem is to compute the determinant of special \textit{Hankel matrices}. Recall that a Hankel matrix (or catalecticant matrix) of a numerical sequence $\C=\{c_i\}$, named after Hermann Hankel, is a matrix defined as
\[H_n^t(\C)=
\begin{bmatrix}
    c_{t} & c_{t+1} & c_{t+2} & \dots  & c_{t+n-1} \\
    c_{t+1} & c_{t+2} & c_{t+3} & \dots  & c_{t+n} \\
    \vdots & \vdots & \vdots & \ddots & \vdots \\
    c_{t+n-1} & c_{t+n} & c_{t+n+1} & \dots  & c_{t+2n-2}
\end{bmatrix}.
\]
In \cite[Theorems 3 and 4]{ck-dy}, the authors use a sequence of ideas to reduce the problem to a previous work of Cigler and Krattenthaler \cite{ci-ck} (the first paper of this series), which describes the Hankel determinants $\det H_n^1(\C)$ and $\det H_n^2(\C)$ of some similar sequences $\C$. Now, consider the sequence $\C$ with elements $1,1,2,5,13,35,96,\ldots $ (see \href{https://oeis.org/A005773}{A005773}). In the following, we suggest the values of the determinant of the Hankel matrix $H_n^0(\C)$
\begin{conjecture}
For positive integers $n$, consider the Hankel matrix
\[H_n^0(\C)=
\begin{bmatrix}
    1 & 1 & 2 & 5&\dots &c_n \\
    1 & 2 & 5 &13&\dots &c_{n+1} \\
    \vdots & \vdots & \vdots & \ddots & \vdots \\
    c_{n} & c_{n+1} & c_{n+2}&c_{n+3} & \dots  & c_{2n}
\end{bmatrix}.
\]
Then
\[
\det H_n^0(\C)=\begin{cases}
0,&n\equiv 3 \pmod 6,\\
-1,&n\equiv 4,5 \pmod 6,\\
1,&n\equiv 2,3 \pmod 6.
\end{cases}
\]
\end{conjecture}
\begin{acknowledgment}
The work of the third author is supported by the Swiss National Science Foundation project 200020-169022 ``Lift and Project Methods for Machine Scheduling Through Theory and Experiments''.
\end{acknowledgment}

\end{document}